\documentclass[11pt]{amsart}
\usepackage[dvips]{graphicx}
\usepackage{amsmath}
\usepackage{amssymb}
\usepackage{amscd}
\usepackage{mathrsfs}
\usepackage{color}
 
\usepackage[all]{xy}%

\DeclareFontEncoding{OT2}{}{}
\DeclareFontSubstitution{OT2}{cmr}{m}{n}

\DeclareFontFamily{OT2}{cmr}{\hyphenchar\font45 }
\DeclareFontShape{OT2}{cmr}{m}{n}{%
   <5><6><7><8><9>gen*wncyr%
   <10><10.95><12><14.4><17.28><20.74><24.88>wncyr10}{}
\DeclareFontShape{OT2}{cmr}{b}{n}{%
   <5><6><7><8><9>gen*wncyb%
   <10><10.95><12><14.4><17.28><20.74><24.88>wncyb10}{}

\DeclareMathAlphabet{\mathcyr}{OT2}{cmr}{m}{n}
\DeclareMathAlphabet{\mathcyb}{OT2}{cmr}{b}{n}
\SetMathAlphabet{\mathcyr}{bold}{OT2}{cmr}{b}{n}

\theoremstyle{plain}
    \newtheorem{theorem}{Theorem}[section]
    \newtheorem{proposition}[theorem]{Proposition}
    \newtheorem{lemma}[theorem]{Lemma}
    \newtheorem{corollary}[theorem]{Corollary}
\theoremstyle{definition}
    \newtheorem{definition}[theorem]{Definition}
    
    \newtheorem{example}[theorem]{Example}
    \newtheorem{remark}[theorem]{Remark}
    
\def\Alphabet{A,B,C,D,E,F,G,H,I,J,K,L,M,N,O,P,Q,R,S,T,U,V,W,X,Y,Z}
\def\alphabet{a,b,c,d,e,f,g,h,i,j,k,l,m,n,o,p,q,r,s,t,u,v,w,x,y,z}
\def\endpiece{xxx}
\def\labelenumi{(\theenumi)}
\def\makeAlphabet[#1]{\expandafter\makeA#1,xxx,}
\def\makealphabet[#1]{\expandafter\makea#1,xxx,}
\def\makeA#1,{\def\temp{#1}\ifx\temp\endpiece\else%
\mkbb{#1}\mkfrak{#1}\mkbf{#1}\mkcal{#1}\mkscr{#1}\expandafter\makeA\fi}%
\def\makea#1,{\def\temp{#1}\ifx\temp\endpiece\else\mkfrak{#1}\mkbf{#1}\expandafter\makea\fi}%
\def\mkbb#1{\expandafter\def\csname bb#1\endcsname{\mathbb{#1}}}
\def\mkfrak#1{\expandafter\def\csname fr#1\endcsname{\mathfrak{#1}}}
\def\mkbf#1{\expandafter\def\csname b#1\endcsname{\mathbf{#1}}}
\def\mkcal#1{\expandafter\def\csname c#1\endcsname{\mathcal{#1}}}
\def\mkscr#1{\expandafter\def\csname s#1\endcsname{\mathscr{#1}}}
\def\makeop[#1]{\xmakeop#1,xxx,}
\def\mkop#1{\expandafter\def\csname #1\endcsname{{\mathrm{#1}}\,}} %
\def\xmakeop#1,{\def\temp{#1}\ifx\temp\endpiece\else\mkop{#1}\expandafter\xmakeop\fi}%
\def\makesymb[#1]{\xmakesymb#1,xxx,}
\def\mksymb#1{\expandafter\def\csname #1\endcsname{{\mathrm{#1}}}} %
\def\xmakesymb#1,{\def\temp{#1}\ifx\temp\endpiece\else\mksymb{#1}\expandafter\xmakeop\fi}%
\makeAlphabet[\Alphabet]
\makealphabet[\alphabet]
\makeop[Spec, Spm, Spf]
\makesymb[pr,id,rank,Ker,Coker,Im,Hom]

\newcommand{\sh}{\mathbin{\mathcyr{sh}}}
\newcommand{\Li}{\mathrm{Li}}
\newcommand{\li}{\mathrm{li}}

\newcommand{\wt}{\mathrm{wt}}
\newcommand{\dep}{\mathrm{dep}}
\newcommand{\blambda}{\boldsymbol\lambda}
\newcommand{\bmu}{\boldsymbol\mu}
\newcommand{\bnu}{\boldsymbol\nu}

\address{Department of Mathematics, Keio University, 3-14-1 Hiyoshi, Kouhoku-ku,Yokohama 223-8522, Japan}
\email{onomath2@gmail.com}

\address{Department of Mathematics, Keio University, 3-14-1 Hiyoshi, Kouhoku-ku,Yokohama 223-8522, Japan}
\email{yamashu@math.keio.ac.jp}

\title{Shuffle product of finite multiple polylogarithms}
\author{Masataka Ono and Shuji Yamamoto}

\begin{document}

\maketitle
\begin{abstract}
In this paper, we define a finite sum analogue of multiple polylogarithms inspired by the work of Kaneko and Zaiger \cite{KZ} 
and prove that they satisfy a certain analogue of the shuffle relation.
Our result is obtained by using a certain partial fraction decomposition due to Komori-Matsumoto-Tsumura \cite{KMT}.
As a corollary, we give an algebraic interpretation of our shuffle product.
 \end{abstract}

\tableofcontents
\setcounter{section}{0}

\section{Introduction}

Multiple polylogarithms
$$
\Li_{k_1, \ldots, k_r}(T)=\sum_{l_1, \ldots, l_r \in \bbZ_{\geq 1}}\frac{T^{l_1+\cdots+l_r}}{l^{k_1}_1(l_1+l_2)^{k_2}\cdots (l_1+\cdots+l_r)^{k_r}} \quad (r \in \bbZ_{\geq1})
$$
and their values at $T=1$, i.e., multiple zeta values (MZVs)
$$
\zeta(k_1, \ldots, k_r)=\Li_{k_1, \ldots, k_r}(1)=\sum_{l_1, \ldots, l_r \in \bbZ_{\geq 1}}\frac{1}{l^{k_1}_1(l_1+l_2)^{k_2}\cdots (l_1+\cdots+l_r)^{k_r}}
$$
are related to many areas including number theory (\cite{DG}), topology (\cite{LM}) and mathematical physics (\cite{BK}), and studied by many authors (for example, \cite{BBBL, BH, Gon}). Here $k_1, \ldots, k_r$ are positive integers, and $k_r$ is larger than 1 for MZVs.
One important property of multiple polylogarithms is the shuffle relation
$$
\Li_\bk(T)\Li_{\bk'}(T)=\Li_{\bk \sh \bk'}(T),
$$
for two finite sequences $\bk=(k_1, \ldots, k_r)$ and $\bk'=(k'_1, \ldots, k'_{r'})$ of positive integers with $k_r, k'_{r'} \; (r, r' \in \bbZ_{\geq1})$.
Here, $\bk \sh \bk'$ is the shuffle product of $\bk$ and $\bk'$ and $\Li_{\bk \sh \bk'}(T)$ 
is the finite sum of multiple polylogarithms corresponding to $\bk \sh \bk'$.
See Definition \ref{shuffle product} for the definition of shuffle products.

In this paper, we introduce a `finite sum version' $\li_\bk(T)$ of multiple polylogarithms 
and prove that they satisfy a certain analogue of shuffle relation.

\begin{definition} \label{ringB}
We define a $\bbQ$-algebra $\cB$ by 
$$
\cB:=\left.\Biggl(\prod_{p:\text{prime}}\bbF_p[T]\Biggr)\right/\Biggl(\bigoplus_{p:\text{prime}}\bbF_p[T]\Biggr).
$$
Thus, an element of $\cB$ is represented by a family $(f_p)_p$ of polynomial $f_p \in \bbF_p[T]$, 
and two families $(f_p)_p$ and $(g_p)_p$ represent the same element of $\cB$ if and only if $f_p=g_p$ for all but finitely many primes $p$.
We also denote such an element of $\cB$ simply by $f_p$ omitting $(\ )_p$ if there is no fear of confusion.
For instance, $T^p$ denotes an element of $\cB$ whose $p$-component is $T^p \in \bbF_p[T]$.
\end{definition}

\begin{definition}[Finite multiple polylogarithm] \label{FMP}
For a positive integer $r$ and $\bk=(k_1, \ldots, k_r) \in (\bbZ_{\geq 1})^r$, 
we define the finite multiple polylogarithm (FMP) $\li_{\bk}(T) \in \cB$ by
\begin{align*}
\li_{\bk}(T)=\li_{k_1, \ldots, k_r}(T)= \begin{cases}
 \displaystyle\sideset{}{'}\sum_{0<l_1, \ldots, l_r<p}\frac{T^{l_1+\cdots+l_r}}{l^{k_1}_1(l_1+l_2)^{k_2}\cdots(l_1+\cdots+l_r)^{k_r}} & (r \geq1), \\
 1 & (r=0).
 \end{cases}
\end{align*}
Here, $\sideset{}{'}\sum$ denotes the sum of fractions whose denominators are prime to $p$.
\end{definition}

Our main theorem is the following:

\begin{theorem}[=Theorem \ref{MainTheorem}]
For non-negative integers $r, r'$ and $\bk=(k_1, \ldots, k_r) \in (\bbZ_{\geq1})^r$, $\bk'=(k'_1, \ldots, k'_{r'}) \in (\bbZ_{\geq1})^{r'}$, 
set $k:=k_1+\cdots+k_r, k':=k'_1+\cdots+k'_{r'}$. Then we have
$$
\li_\bk(T)\li_{\bk'}(T) \equiv \li_{\bk \sh \bk'}(T) \pmod{\cR_{k+k', k+k'-1}}.
$$
Moreover, $\li_{\bk}(T)\li_{\bk'}(T)-\li_{\bk \sh \bk'}(T) \in \cR_{k+k', k+k'-1}$ can be calculated explicitly from $\li_\bk(T)$ and $\li_{\bk'}(T)$.
\end{theorem}
Here, $\li_{\bk \sh \bk'}(T)$ is the finite sum of FMPs corresponding to $\bk \sh \bk'$. 
See Definition \ref{shuffle product} for the definition of FMP corresponding to $\bk \sh \bk'$.
$\cR_{a, b}$ is a certain $\bbQ$-vector subspace of $\cB$ defined in Definition \ref{filtration}.

One of our motivations is to generalize the studies on the finite analogue of polylogarithms 
by Kontsevich \cite{Kon}, Elbaz-Vincent and Gangl \cite{EG} and Besser \cite{Bes} to the multiple case.
Their finite polylogarithm is defined as
$$
\cL_k(T):=\sum_{0<n<p}\frac{T^n}{n^k} \in \bbF_p[T]
$$
for each prime $p$. 
The idea to work in the ring $\cB$ is inspired by the definition of finite multiple zeta values introduced by Kaneko and Zagier \cite{KZ}:

\begin{definition}
\begin{enumerate}
\renewcommand{\labelenumi}{(\roman{enumi})}
\item A $\bbQ$-algebra $\cA$ is defined as
$$
\cA:=\left.\Biggl(\prod_{p;\text{prime}}\bbF_p\Biggr)\right/\Biggr(\bigoplus_{p;\text{prime}}\bbF_p\Biggl).
$$
As in the case of $\cB$, we omit $(\ )_p$ to represent an element of $\cA$.
\item
For positive integer $r$ and $\bk=(k_1, \ldots, k_r) \in (\bbZ_{\geq1})^r$,
we define the finite multiple zeta value (FMZV) $\zeta_\cA(\bk) \in \cA$ by
\begin{align*}
\zeta_\cA(\bk)=\zeta_\cA(k_1, \ldots, k_r):=
  \begin{cases}
  \displaystyle\sum_{\substack{l_1, \ldots, l_r \in \bbZ_{\geq1} \\ l_1+\cdots+l_r<p}}\frac{1}{l^{k_1}_1\cdots (l_1+\cdots+l_r)^{k_r}} & (r \geq1), \\
  1 & (r=0).
  \end{cases}
\end{align*}
\end{enumerate}
\end{definition}

Taking the usual complex case into account, 
finite analogues of multiple polylogarithms should coincide with the corresponding FMZVs at $T=1$. 
Unfortunately, our definition of FMPs does not give FMZVs at $T=1$ (conjecturally, at least).
Indeed, we can prove that $\li_\bk(1)=0$, so our definition is unsatisfactory in this sense.
It is important, however, that our FMPs satisfy the shuffle relation.
We hope that the study in this paper will be a base to find a better definition.

The contents of this paper are as follows. 
In section 2, we will introduce a variant $\zeta^{(i)}_\cA(k_1, \ldots, k_r)$ of FMZVs  
and prove that $\zeta^{(i)}_\cA(k_1, \ldots, k_r)$ is expressed as a sum of FMZVs of the same weight.
In section 3, we will prove the main theorem (Theorem \ref{MainTheorem}).
In the proof, we use the method of partial fraction decomposition (Lemma \ref{PFD}) due to [KMT].
In section 4, we will give an algebraic interpretation of our main theorem using the ``stuffle-shuffle" algebra.

{\ }\\
$\underline{\textbf{Notation}}$\\
For non negative integers $a, b, c$ and $l_1, \ldots, l_a , m_1, \ldots, m_b, n_1, \ldots, n_c$,
we set new positive integers $L_i, M_j, N_k$:
\begin{equation} \label{eq:1}
\begin{cases}
 L_i:=l_1+\cdots+l_i   &(1 \leq i \leq a), \\
 M_j:=m_1+\cdots+m_j  &(1 \leq j \leq b), \\
 N_k:=n_1+\cdots+n_k & (1\leq k \leq c).
\end{cases}
\end{equation}
Let $I$ denote the following set: 
$$
I=\coprod_{r\in \bbZ_{\geq0}}I_r, \quad I_r:=(\bbZ_{\geq1})^r.
$$
An element in $I$ is called an index. For an index $\bk=(k_1, \ldots, k_r)$, 
the integers $k_1+\cdots+k_r$ and $r$ are called the weight of $\bk$ and depth of $\bk$ respectively,
and denoted by $\wt(\bk)$ and $\dep(\bk)$.
The unique index in $I_0$ is denoted by $\emptyset$, for which we understand $\wt(\emptyset)=\dep(\emptyset)=0$.

For two indices $\boldsymbol\lambda=(\lambda_1, \ldots, \lambda_a), \boldsymbol\mu=(\mu_1, \ldots, \mu_b) \in I$, 
we define two new indices $\boldsymbol\lambda \bullet \boldsymbol\mu$ and $\boldsymbol\lambda \star \boldsymbol\mu$ by
$$
\boldsymbol\lambda \bullet \boldsymbol\mu:=(\lambda_1, \ldots, \lambda_a, \mu_1, \ldots, \mu_b), \; 
\boldsymbol\lambda \star \boldsymbol\mu:=(\lambda_1, \ldots, \lambda_{a-1}, \lambda_a+\mu_b, \mu_{b-1}, \ldots, \mu_1).
$$
For a positive integer $r$, we set $[r]:=\{1, \ldots, r\}$.
Finally, for a non negative integer $k$, we define the $\bbQ$-vector subspace $\cZ_{\cA, k}$ of $\cA$ by
$$
\cZ_{\cA, k}:=
\langle \zeta_\cA(\bk) \mid \bk \in I, \wt(\bk)=k \rangle_\bbQ 
$$
and set $\cZ_\cA:=\sum_{k=0}^\infty \cZ_{\cA, k}$.
Note that $\cZ_\cA$ forms a $\bbQ$-algebra by the stuffle relation of FMZVs.

\section{A variant of finite multiple zeta values}

In this section, we define a variant $\zeta^{(i)}_\cA$ of FMZVs.
This variant turns out to be a sum of FMZVs (Proposition \ref{variant}), 
and will play an important role in the proof of the main theorem.

\begin{definition} \label{variant of FMZV}
For an index $\bk=(k_1, \ldots, k_r)$ in $I$ and $1 \leq i \leq r$, we define a variant of FMZVs as an element in $\cA$ by
$$
\zeta^{(i)}_\cA(\bk)
:=\sideset{}{'}\sum_{\substack{0<l_1, \ldots, l_r <p \\ (i-1)p< L_r < ip}}
\frac{1}{L^{k_1}_1\cdots L^{k_r}_r}.
$$
\end{definition}

Note that $\zeta^{(1)}_\cA(\bk)$ coincides with $\zeta_\cA(\bk)$ by definition.

\begin{remark}
By setting $l'_i:=l_i-p$, we see that $\zeta^{(r)}_\cA(\bk) = (-1)^{\wt(\bk)}\zeta_\cA(\bk)$.
\end{remark}

To explain that $\zeta^{(i)}_\cA$ can be expressed as a sum of $\zeta_\cA$'s, 
we introduce more notation.
For a positive integers $r$ and $s$, set
$$
X_r:=\{(l_1, \ldots, l_r) \in [p-1]^r \mid (L_1, p)=(L_2, p)=\cdots=(L_r, p)=1\},
$$
$$
\Phi_r:=\bigsqcup_{s=1}^r\Phi_{r,s},\quad \Phi_{r,s}:=\{\phi : [r]\rightarrow [s] : \text{surjective}\mid \phi(a)\neq \phi(a+1) \text{ for all } a \in [r-1]\},
$$
$$
Y_s := \{(A_1, \ldots, A_s) \in [p-1]^s \mid 0<A_1<\cdots<A_s<p\}
$$
and we define an integer $\delta_\phi(i)$ by 
\begin{equation} \label{eq:2}
\delta_\phi(i):=\#\{a \in [i-1] \mid \phi(a)>\phi(a+1)\} \quad (1\leq i \leq r)
\end{equation}
for $\phi \in \Phi_r$.
Next, for $x \in X_r$, there exist 
$$
s \in [r], \quad \phi=\phi_x \in \Phi_{r, s}, \quad(A_1, \ldots, A_s) \in Y_s
$$
uniquely satisfying that $l_1+\cdots+l_i \equiv A_{\phi(i)} \pmod{p}$ for $i=1, \ldots, r$.
Moreover, using $s, \phi$ and $(A_1, \ldots, A_s)$ above, 
we define the map $f : X_r \rightarrow \bigsqcup_{s=1}^r(\Phi_{r,s} \times Y_s)$ by sending $x \in X_r$ to $(\phi, (A_1, \ldots, A_s))$.

\begin{lemma} \label{surjective}
$f$ is a bijection.
\end{lemma}

\begin{proof}
We construct the inverse map of $f$.
For any $\phi \in \Phi_{r,s}, (A_1, \ldots, A_s) \in Y_s$ and $i \in [r]$, we define an integer $l_i$ as follows:
         \begin{equation} \label{eq:3}
         l_i:=
         \begin{cases}
         A_{\phi(1)} & (i=1), \\
         (A_{\phi(i)}+\delta_\phi(i)p)-(A_{\phi(i-1)}+\delta_\phi(i-1)p) & (2\leq i \leq r).
         \end{cases}
         \end{equation}
         Since
         \begin{equation} \label{eq:4}
         \delta_\phi(i)-\delta_\phi(i-1)=
         \begin{cases}
         0 & \text{if $\phi(i-1)<\phi(i)$}, \\
         1 & \text{if $\phi(i-1)>\phi(i)$}, 
         \end{cases}
         \end{equation}
         we have $0<l_i<p$ for all $i$. Moreover, since
         \begin{align*}
         l_1+\cdots+l_i=&A_{\phi(1)} + (A_{\phi(2)}+\delta_\phi(2)p)-A_{\phi(1)} \\
                         &+(A_{\phi(3)}+\delta_\phi(3)p)-(A_{\phi(2)}+\delta_\phi(2)p) \\
                         &+\cdots +(A_{\phi(i)}+\delta_\phi(i)p) -(A_{\phi(i-1)}+\delta_\phi(i-1)p)\\
                       =&A_{\phi(i)}+\delta_\phi(i)p,
         \end{align*}
         the remainder of $l_1+\cdots+l_i$ modulo $p$ is $A_{\phi(i)}$ and $l_1+\cdots+l_i$ is prime to $p$.
         Thus we obtain a map $g : \bigsqcup_{s=1}^r(\Phi_{r,s} \times Y_s) \rightarrow X_r$ by $g(\phi, (A_1, \ldots, A_s))=(l_1, \ldots, l_r)$.
         We can easily prove that $g$ is the inverse of $f$ by \eqref{eq:2} and \eqref{eq:3}. 
\end{proof}

Next we define two maps
$$
\alpha : X_r \rightarrow [r], \quad \beta : \Phi_r \rightarrow [r]
$$
as follows:
$$
\text{$\alpha(l_1, \ldots, l_r)$ is defined to be the unique integer $n$ satisfying  $(n-1)p<l_1+\cdots+l_r<np$},
$$
$$
\beta(\phi):=\delta_\phi(r)+1.
$$
Using $\alpha$ and $\beta$, for $1 \leq i\leq r$, we set
$$
\text{$X^i_r:=\alpha^{-1}(i)$, \quad $\Phi^i_r:=\beta^{-1}(i)$},
$$
and $X_\phi:=\{x \in X_r \mid \phi_x=\phi\}$ for $\phi \in \Phi_r$. Then we have
\begin{equation} \label{eq:5}
X^i_r=\bigsqcup_{\phi \in \Phi^i_r}X_\phi
\end{equation}
for $1\leq i \leq r$. Further, for $\phi \in \Phi^i_{r, s}:=\Phi^i_r \cap \Phi_{r, s}$,
the composition 
\begin{equation} \label{eq:6}
X_\phi \xrightarrow{f} \{\phi\} \times Y_s \xrightarrow{\pr_2} Y_s
\end{equation}
is a bijection.

The following is the main result in this section.

\begin{proposition} \label{variant}
For $1 \leq i \leq r$ and an index $\bk=(k_1, \dots, k_r)$ in $I_r$, we have
$$
\zeta^{(i)}_\cA(\bk)=\sum_{\phi \in \Phi^i_r}\zeta_\cA\left(\sum_{\phi(j)=1}k_j, \ldots, \sum_{\phi(j)=s}k_j \right).
$$
\end{proposition}

\begin{proof}
By the definition of $X^i_r$, \eqref{eq:5} and \eqref{eq:6}, we obtain 
\begin{align*}
\zeta^{(i)}_\cA(k_1, \ldots, k_r)
&=\sum_{(l_1, \ldots, l_r) \in X^i_r}\frac{1}{l^{k_1}_1(l_1+l_2)^{k_2} \cdots (l_1+\cdots+l_r)^{k_r}} \; (\text{definition of $X^i_r$}) \\
&=\sum_{\phi \in \Phi^i_r}\sum_{(l_1, \ldots, l_r) \in X_{\phi }}\frac{1}{l^{k_1}_1(l_1+l_2)^{k_2} \cdots (l_1+\cdots+l_r)^{k_r}} \; (\text{by \eqref{eq:5}})\\
&=\sum_{\phi \in \Phi^i_r}\sum_{0<A_1<\cdots <A_s<p}\frac{1}{A^{k_1}_{\phi (1)} \cdots A^{k_r}_{\phi (r)}} \; (\text{by \eqref{eq:6}})\\
&=\sum_{\phi \in \Phi^i_r}\sum_{0<A_1<\cdots <A_s<p}\frac{1}{A^{\sum_{\phi (j)=1}k_j}_1 \cdots A^{\sum_{\phi (j)=s}k_j}_s}\\
&=\sum_{\phi \in \Phi^i_r}\zeta_\cA\left(\sum_{\phi(j)=1}k_j, \ldots, \sum_{\phi(j)=s}k_j \right). \qedhere
\end{align*}
\end{proof}

\begin{example}
We present non-trivial examples of $\zeta^{(i)}_{\cA}(k_1, \ldots, k_r)$ for $r=3, 4$.
\begin{enumerate}

\item $\zeta^{(2)}_{\cA}(k_1, k_2, k_3)
         =\zeta_\cA(k_3, k_1, k_2)+\zeta_\cA(k_1, k_3, k_2)+\zeta_\cA(k_2, k_3, k_1)+\zeta_\cA(k_2, k_1, k_3)
            + \zeta_\cA(k_1+k_3, k_2)+\zeta_\cA(k_2, k_1+k_3)$.
\item 
         \begin{align*}
           &\zeta^{(2)}_{\cA}(k_1, k_2, k_3, k_4) \\
         =&\zeta_\cA(k_4, k_1, k_2, k_3)+\zeta_\cA(k_3, k_4, k_1, k_2)+\zeta_\cA(k_2, k_3, k_4, k_1)+\zeta_\cA(k_1, k_4, k_2, k_3) \\
           &+\zeta_\cA(k_3, k_1, k_4, k_2)+\zeta_\cA(k_2, k_3, k_1, k_4)+\zeta_\cA(k_1, k_2, k_4, k_3)+\zeta_\cA(k_3, k_1, k_2, k_4) \\
           &+\zeta_\cA(k_2, k_1, k_3, k_4)+\zeta_\cA(k_1, k_3, k_4, k_2)+\zeta_\cA(k_1, k_3, k_2, k_4) + \zeta_\cA(k_1+k_3, k_2, k_4) \\
           &+\zeta_\cA(k_2, k_1+k_3, k_4)+\zeta_\cA(k_1+k_3, k_4, k_2)+\zeta_\cA(k_1+k_4, k_2, k_3)+\zeta_\cA(k_3, k_1+k_4, k_2) \\
           &+\zeta_{\cA}(k_2, k_3, k_1+k_4)+\zeta_\cA(k_1, k_2+k_4, k_3)+\zeta_\cA(k_3, k_1, k_2+k_4)+\zeta_\cA(k_1, k_3, k_2+k_4)\\
           &+\zeta_\cA(k_1+k_3, k_2+k_4).
         \end{align*}
         \begin{align*}
           &\zeta^{(3)}_{\cA}(k_1, k_2, k_3, k_4) \\
         =&\zeta_\cA(k_3, k_2, k_1, k_4)+\zeta_\cA(k_2, k_1, k_4, k_3)+\zeta_\cA(k_1, k_4, k_3, k_2)+\zeta_\cA(k_3, k_2, k_4, k_1) \\
           &+\zeta_\cA(k_2, k_4, k_1, k_3)+\zeta_\cA(k_4, k_1, k_3, k_2)+\zeta_\cA(k_3, k_4, k_2, k_1)+\zeta_\cA(k_4, k_2, k_1, k_3) \\
           &+\zeta_\cA(k_4, k_3, k_1, k_2)+\zeta_\cA(k_2, k_4, k_3, k_1)+\zeta_\cA(k_4, k_2, k_3, k_1)+\zeta_\cA(k_4, k_2, k_1+k_3) \\
           &+\zeta_\cA(k_4, k_1+k_3, k_2)+\zeta_\cA(k_2, k_4, k_1+k_3)+\zeta_\cA(k_3, k_2, k_1+k_4)+\zeta_\cA(k_2, k_1+k_4, k_3) \\
           &+\zeta_{\cA}( k_1+k_4, k_3, k_2)+\zeta_\cA(k_3, k_2+k_4, k_1)+\zeta_\cA(k_2+k_4, k_1, k_3)+\zeta_\cA(k_2+k_4, k_3, k_1)\\
           &+\zeta_\cA(k_2+k_4, k_1+k_3).
         \end{align*}
         
\end{enumerate}

\end{example}

\section{Proof of the main theorem}

In this section, we define FMPs $\li(\blambda, \bmu, \bnu; T)$ of type $(\blambda, \bmu, \bnu)$ and prove the main theorem using a method inspired by [KMT]. 
 $\li(\blambda, \bmu, \bnu; T)$ is a generalization of both FMPs and products of two FMPs.

\begin{definition}
For indices $\blambda = (\lambda_1, \ldots, \lambda_a), \bmu =(\mu_1, \ldots, \mu_b)$ and $\bnu=(\nu_1, \ldots, \nu_c)$ in $I$ ($a, b, c \in \bbZ_{\geq0}$), 
we define a FMP of type $(\blambda, \bmu, \bnu)$ by
$$
\li(\blambda, \bmu, \bnu; T)
:=\sideset{}{'}\sum_{\substack{0<l_1, \ldots, l_a <p \\ 0<m_1, \ldots, m_b <p \\0<n_1, \ldots, n_c <p}}
\frac{T^{L_a+M_b+N_c}}{
\displaystyle\prod_{x=1}^a L^{\lambda_x}_x
\displaystyle\prod_{y=1}^b M^{\mu_y}_y
\displaystyle\prod_{z=1}^c \left(L_a+M_b+N_z\right)^{\nu_z}
}
\in \cB.
$$
Recall the definitions of $L_x, M_y$ and $N_z$ defined in \eqref{eq:1}.
\end{definition}

\begin{remark}
By the definition of $\li(\blambda, \bmu, \bnu; T)$, we have
$$
\li(\blambda, \emptyset, \emptyset; T)=\li(\emptyset, \blambda, \emptyset; T)=\li(\emptyset, \emptyset, \blambda; T)=\li_{\blambda}(T), 
$$
$$
\li(\blambda, \emptyset, \bmu; T)=\li(\emptyset, \blambda, \bmu; T)
=\li_{\blambda \bullet \bmu}(T), 
$$
and
$$
\li(\blambda, \bmu, \emptyset; T)=\li_{\blambda}(T)\li_{\bmu}(T).
$$
\end{remark}

Next, to state our main theorem, we introduce $\cZ_\cA[T^p]$-submodule
$$
\cR:=\sum_{\bk \in I} \cZ_\cA[T^p] \li_{\bk}(T) 
$$
of $\cB$, and $\bbQ$-subspaces $\cR\supset \cR_a \supset \cR_{a, b}$ as follows. 

\begin{definition} \label{filtration}
For a non-negative integer $a$, we define a $\bbQ$-subspace  
$$
\cR_a:=\langle \zeta_\cA(\bk)\cdot(T^p)^n\cdot\li_{\bk'}(T) \mid n \in \bbZ_{\geq0}, \wt(\bk)+\wt(\bk')=a \rangle_\bbQ
$$
of $\cR$. Then we have $\cR=\sum_{a=0}^\infty \cR_a$.
Moreover, for non-negative integers $a$ and $b \in \{0, \ldots, a\}$, we define a $\bbQ$-subspace
$$
\cR_{a, b}:=\langle \zeta_\cA(\bk)\cdot(T^p)^n\cdot\li_{\bk'}(T) \mid n \in \bbZ_{\geq0}, \wt(\bk)+\wt(\bk')=a, \wt(\bk')\leq b\rangle_\bbQ
$$
of $\cR_a$. Then we have an increasing filtration
$$
\cR_{a, 0} \subset \cR_{a, 1} \subset \cdots \subset \cR_{a, a-1} \subset \cR_{a, a} =\cR_a
$$
of $\bbQ$-vector spaces.
\end{definition}

To define the shuffle product $\bk \sh \bk'$ of indices $\bk$ and $\bk'$, we prepare some terminologies.
Let $\frH$ be the noncommutative polynomial ring $\bbQ \langle x, y \rangle$ of variables $x$ and $y$ over $\bbQ$, 
and we set
$$
\frH^1:=\bbQ + \frH y, \; \frH^0:=\bbQ + x\frH y.
$$
Note that $\frH^1$ is generated by $z_k:=x^{k-1}y \; (k=1, 2, \cdots)$ as a $\bbQ$-algebra.

\begin{definition} \label{shuffle algebra}
We define the shuffle product $\sh : \frH \times \frH \rightarrow \frH$ on $\frH$ by the following rule and $\bbQ$-bilinearity.

\begin{enumerate}
\renewcommand{\labelenumi}{(\roman{enumi})}

\item $w \sh 1=1 \sh w =w$ for all $w \in \frH$.

\item $(u_1w_1)\sh(u_2w_2)=u_1(w_1\sh u_2w_2)+u_2(u_1w_1\sh w_2)$ for all $w_1, w_2 \in \frH$ and $u_1, u_2 \in \{x, y\}$.

\end{enumerate}

For instance, we have
$$
z_2 \sh z_3=z_2z_3+3z_3z_2+6z_4z_1.
$$
\end{definition}

\begin{definition} \label{shuffle product}
For $\bk=(k_1, \ldots, k_r) \in I$, we set $z_{\bk}:=z_{k_1}\cdots z_{k_r}$.
We define the shuffle product $\bk\sh\bk'$ of indices $\bk$ and $\bk'$ as the formal sum of indices 
corresponding to $z_{\bk} \sh z_{\bk'}$. For instance, 
$$
(2)\sh(3)=(2, 3)+3(3, 2)+6(4,1).
$$
\end{definition}

Further, for indices $\bk$ and $\bk'$, we define $\Li_{\bk \sh \bk'}(T)$ as the finite linear sum of multiple polylogarithms with indices corresponding to $\bk \sh \bk'$.
For instance, since $(2)\sh(3)=(2,3)+3(3,2)+6(4,1)$, 
$$
\Li_{(2)\sh(3)}(T):=\Li_{2,3}(T)+3\Li_{3,2}(T)+6\Li_{4,1}(T).
$$
Also $\li_{\bk \sh \bk'}(T)$ is defined in the same way.

The main theorem of this paper is the following.

\begin{theorem} \label{MainTheorem}
For indices $\bk=(k_1, \ldots, k_r), \bk'=(k'_1, \ldots, k'_{r'})$ in $I$ with $k:=\wt(\bk)$ and $k':=\wt(\bk')$, 
we have 
$$
\li_\bk(T)\li_{\bk'}(T) \equiv \li_{\bk \sh \bk'}(T) \pmod{\cR_{k+k', k+k'-1}}.
$$
\end{theorem}

To prove Theorem \ref{MainTheorem}, we first show the following proposition.

\begin{proposition} \label{key prop}
Assume that $\blambda, \bmu \neq \emptyset$. Then we have the following equality.
\begin{align} \label{eq:7}
&\li(\blambda, \bmu, \bnu; T)\\
=&\sum_{\tau=0}^{\mu_b-1}\binom{\lambda_a-1+\tau}{\tau}
     \li((\lambda_1, \ldots, \lambda_{a-1}), (\mu_1, \ldots, \mu_{b-1}, \mu_b-\tau), (\lambda_a+\tau) \bullet \boldsymbol{\nu}; T)\nonumber\\
  &+\sum_{\tau=0}^{\lambda_a-1}\binom{\mu_b-1+\tau}{\tau}
     \li((\lambda_1, \ldots, \lambda_{a-1}, \lambda_a-\tau), (\mu_1, \ldots, \mu_{b-1}), (\mu_b+\tau) \bullet \boldsymbol{\nu}; T)\nonumber\\
  &+(-1)^{\wt(\boldsymbol{\mu})}\left(\sum_{i=1}^{a+b-1}
     \zeta^{(i)}_\cA(\blambda \star \bmu)(T^p)^i\right)\li_{\boldsymbol{\nu}}(T)\nonumber\\
  &- \binom{\lambda_a+\mu_b-1}{\lambda_a}
     \left(\sum_{j=1}^{a-1}\zeta^{(j)}_\cA(\lambda_1, \ldots, \lambda_{a-1})(T^p)^j\right)
     \li_{\mu_1, \ldots, \mu_{b-1}, \lambda_a+\mu_b, \boldsymbol{\nu}}(T)\nonumber\\
  &- \binom{\lambda_a+\mu_b-1}{\mu_b}
      \left(\sum_{j=1}^{b-1}\zeta^{(j)}_\cA(\mu_1, \cdots, \mu_{b-1})(T^p)^j\right)
      \li_{\lambda_1, \ldots, \lambda_{a-1}, \lambda_a+\mu_b, \boldsymbol{\nu}}(T)\nonumber.
\end{align}
\end{proposition}

The next lemma plays the most important role in the proof of Proposition \ref{key prop}.

\begin{lemma}[{[KMT]} equation(26)] \label{PFD}
For indeterminates $X$ and $Y$ and positive integers $\alpha$ and $\beta$, 
we have the following partial fraction decomposition;
$$
\frac{1}{X^\alpha Y^\beta}=\sum_{\tau=0}^{\beta-1}\binom{\alpha-1+\tau}{\tau}\frac{1}{(X+Y)^{\alpha+\tau}Y^{\beta-\tau}}
                                           +\sum_{\tau=0}^{\alpha-1}\binom{\beta-1+\tau}{\tau}\frac{1}{(X+Y)^{\beta+\tau}X^{\alpha-\tau}}.
$$
\end{lemma}

\begin{proof}[Proof of Proposition \ref{key prop}]

First, we separate $\li(\boldsymbol{\lambda}, \boldsymbol{\mu}, \boldsymbol{\nu}; T)$ into 
two parts according to that  $L_a+M_b$ is prime to $p$ or not: 
\begin{equation} \label{eq:8}
\li(\boldsymbol{\lambda}, \boldsymbol{\mu}, \boldsymbol{\nu}; T)
=\Biggl(\sideset{}{'}\sum_{\substack{0<l_1, \ldots, l_a<p \\ 0<m_1, \ldots, m_b<p \\ 0<n_1, \ldots, n_c<p \\ p \nmid L_a+M_b}} 
   + \sideset{}{'}\sum_{\substack{0<l_1, \ldots, l_a<p \\ 0<m_1, \ldots, m_b<p \\ 0<n_1, \ldots, n_c<p \\ p\mid L_a+M_b}}\Biggr)
   \frac{T^{L_a+M_b+N_c}}
{\displaystyle\prod_{x=1}^a L^{\lambda_x}_x
\displaystyle\prod_{y=1}^b M^{\mu_y}_y
\displaystyle\prod_{z=1}^c \left(L_a+M_b+N_z\right)^{\nu_z}}.
\end{equation}
The second term in \eqref{eq:8} is calculated as 
\begin{align*}
&\sum_{i=1}^{a+b-1}\sideset{}{'}\sum_{\substack{0<l_1, \ldots, l_a<p \\ 0<m_1, \ldots, m_b<p \\ 0<n_1, \ldots, n_c<p \\ L_a+M_b=ip}}
\frac{T^{L_a+M_b+N_c}}
{
\displaystyle\prod_{x=1}^a L^{\lambda_x}_x
\displaystyle\prod_{y=1}^b M^{\mu_y}_y
\displaystyle\prod_{z=1}^c \left(L_a+M_b+N_z\right)^{\nu_z}
}\\
=&(-1)^{\wt(\boldsymbol{\mu})}\sum_{i=1}^{a+b-1}(T^p)^i
\sideset{}{'}\sum_{\substack{0<l_1, \ldots, l_a<p \\ 0<m_2, \ldots, m_b<p \\ 0<n_1, \ldots, n_c<p \\ (i-1)p<L_a+m_b+\cdots+m_2<ip}}
\frac{T^{N_c}}
{
\displaystyle\prod_{x=1}^{a-1} L^{\lambda_x}_x L^{\lambda_a+\mu_b}_a
\displaystyle\prod_{y=1}^{b-1}(L_a+m_b+\cdots+m_{y+1})^{\mu_y}
\displaystyle\prod_{z=1}^c N^{\nu_z}_z
},
\end{align*}
by using the congruences $L_a+m_b+\cdots+m_{y+1} \equiv -M_y \pmod{p}\; (1 \leq y \leq b-1)$.
If we regard that $l_{a+1}:=m_b, \ldots, l_{a+b-1}:=m_2$, 
we see that the second term in \eqref{eq:8} coincides with
$$
(-1)^{\wt(\boldsymbol{\mu})}\left(\sum_{i=1}^{a+b-1}
\zeta^{(i)}_\cA(\boldsymbol\lambda \star \boldsymbol\mu)(T^p)^i\right)
\li_{\boldsymbol{\nu}}(T).
$$
Next, using Lemma \ref{PFD} for $(X, Y)=(L_a, M_b)$ and $(\alpha, \beta)=(\lambda_a, \mu_b)$, 
we calculate the first term in \eqref{eq:8} as follows:
\begin{align}
&\sum_{\tau=0}^{\mu_b-1}\binom{\lambda_a-1+\tau}{\tau}
\sideset{}{'}\sum_{\substack{0<l_1, \ldots, l_a<p \\0<m_1, \ldots, m_b<p \\0<n_1, \ldots, n_c<p \\ p \nmid L_a}}
\frac{T^{L_a+M_b+N_c}}
{
\displaystyle\prod_{x=1}^{a-1} L^{\lambda_x}_x
\displaystyle\prod_{y=1}^{b-1} M^{\mu_y}_y
M^{\mu_b-\tau}_b(L_a+M_b)^{\lambda_a+\tau}
\displaystyle\prod_{z=1}^c \left(L_a+M_b+N_z\right)^{\nu_z}
} \label{eq:9} \\
&+\sum_{\tau=0}^{\lambda_a-1}\binom{\mu_b-1+\tau}{\tau}
\sideset{}{'}\sum_{\substack{0<l_1, \ldots, l_a<p \\0<m_1, \ldots, m_b<p \\0<n_1, \ldots, n_c<p \\p\nmid M_b}}
\frac{T^{L_a+M_b+N_c}}
{
\displaystyle\prod_{x=1}^{a-1} L^{\lambda_x}_x
L^{\lambda_a-\tau}_a
\displaystyle\prod_{y=1}^{b-1} M^{\mu_y}_y
(L_a+M_b)^{\mu_b+\tau}
\displaystyle\prod_{z=1}^c \left(L_a+M_b+N_z\right)^{\nu_z}
}. \label{eq:10}
\end{align}
Here, \eqref{eq:9} is separated as follows:
$$
\sideset{}{'}\sum_{\substack{0<l_1, \ldots, l_a<p \\ 0<m_1, \ldots, m_b<p \\ 0<n_1, \ldots, n_c<p \\ p\nmid L_a,}} =
\sideset{}{'}\sum_{\substack{0<l_1, \ldots, l_a<p \\ 0<m_1, \ldots, m_b<p \\ 0<n_1, \ldots, n_c<p }}
- \sideset{}{'}\sum_{\substack{0<l_1, \ldots, l_a<p \\ 0<m_1, \ldots, m_b<p \\ 0<n_1, \ldots, n_c<p \\p \mid L_a}}.
$$
The first sum in the right hand side coincides with 
$\li((\lambda_1, \ldots, \lambda_{a-1}), (\mu_1, \ldots, \mu_{b-1}, \mu_b-\tau), (\lambda_a+\tau) \bullet \boldsymbol{\nu}; T)$.
The second can be calculated as follows:

\begin{align*} 
& \sum_{j=1}^{a-1}\sideset{}{'}\sum_{\substack{0<l_1, \ldots, l_a<p \\ 0<m_1, \ldots, m_b<p \\ 0<n_1, \ldots, n_c<p \\ L_a=jp}} 
\frac{T^{L_a+M_b+N_c}}
{
\displaystyle\prod_{x=1}^{a-1} L^{\lambda_x}_x
\displaystyle\prod_{y=1}^{b-1} M^{\mu_y}_y
M^{\mu_b-\tau}_b (L_a+M_b)^{\lambda_a+\tau}
\displaystyle\prod_{z=1}^c \left(L_a+M_b+N_z\right)^{\nu_z}
}\\
=&\sum_{j=1}^{a-1}(T^p)^j\sideset{}{'}\sum_{\substack{0<l_1, \ldots, l_{a-1}<p \\ 0<m_1, \ldots, m_b<p \\ 0<n_1, \ldots, n_c<p \\ (j-1)p<L_{a-1}<jp}} 
\frac{T^{M_b+N_c}}
{
\displaystyle\prod_{x=1}^{a-1} L^{\lambda_x}_x
\displaystyle\prod_{y=1}^{b-1} M^{\mu_y}_y
M^{\lambda_a+\mu_b}_b
\displaystyle\prod_{z=1}^c \left(M_b+N_z\right)^{\nu_z}
}\\
=& \left(\sum_{j=1}^{a-1}
\zeta^{(j)}_\cA(\lambda_1, \ldots, \lambda_{a-1})(T^p)^j\right)\li_{\mu_1, \ldots, \mu_{b-1}, \lambda_a+\mu_b, \boldsymbol{\nu}}(T).
\end{align*}
By the same calculation for \eqref{eq:10}, we obtain the desired formula.
\end{proof}

\begin{proof}[Proof of Theorem \ref{MainTheorem}]
By Proposition \ref{variant}, we see that the variant $\zeta^{(i)}_\cA$ of FMZV is contained in $\cZ_\cA$.
Further, note that all terms in \eqref{eq:7} have total weight $w := \wt(\blambda) + \wt(\bmu) + \wt(\bnu)$, 
and all the 3rd, 4th and 5th terms in the right hand side belong in \eqref{eq:7} to $\cR_{w, w-1}$.
Hence we have
\begin{align} \label{eq:11}
\li(\blambda, \bmu, \bnu; T)\equiv&\sum_{\tau=0}^{\mu_b-1}\binom{\lambda_a-1+\tau}{\tau}
     \li((\lambda_1, \ldots, \lambda_{a-1}), (\mu_1, \ldots, \mu_{b-1}, \mu_b-\tau), (\lambda_a+\tau) \bullet \bnu; T)\\
  &+\sum_{\tau=0}^{\lambda_a-1}\binom{\mu_b-1+\tau}{\tau}
     \li((\lambda_1, \ldots, \lambda_{a-1}, \lambda_a-\tau), (\mu_1, \ldots, \mu_{b-1}), (\mu_b+\tau) \bullet \bnu; T)\nonumber
\end{align}
in $\cR_w/\cR_{w, w-1}$. 
Therefore, the main theorem is obtained from the same argument in \cite{KMT}.
Namely, consider the $\bbQ$-linear map
\begin{equation*}
Z : \frH \longrightarrow \cR; \quad z_{\bk} \longmapsto \li_{\bk}(T).
\end{equation*}
Then by induction on $\dep(\blambda)+\dep(\bmu)$, we can prove 
$Z((z_{\blambda}\sh z_{\bmu})z_{\bnu}) \equiv \li(\blambda, \bmu, \bnu; T)$ in $\cR_w/\cR_{w, w-1}$ by using \eqref{eq:11}. 
Thus we can prove $\li_{\blambda}(T)\li_{\bmu}(T)\equiv \li_{\blambda\sh\bmu}(T) \pmod{\cR_{w,w-1}}$ by equation (17) in \cite{KMT}.
This completes the proof of Theorem \ref{MainTheorem}.
\end{proof}

The following is a corollary of Theorem \ref{MainTheorem}.

\begin{corollary} \label{algebra}
$\cR$ forms a $\cZ_\cA[T^p]$-subalgebra of $\cB$.
\end{corollary}

\begin{remark}
Let $\cP$ be the $\bbQ$-vector space generated by the multiple polylogarithms:
$$
\cP:=\sum_{\bk \in I}\bbQ\Li_{\bk}(T).
$$ 
Then Corollary \ref{algebra} may be regarded as a finite analogue of the fact that
$\cP$ is a $\bbQ$-algebra, which is a consequence of the shuffle relation.
\end{remark}

\section{An algebraic interpretation}

In this section, we will give an algebraic interpretation of our main theorem (Theorem \ref{MainTheorem}).
Note that by Corollary \ref{algebra}, we have 
\begin{equation} \label{eq:12}
\cR_{a_1, b_1}\cdot\cR_{a_2, b_2} \subset \cR_{a_1+a_2, b_1+b_2}
\end{equation}
for all $a_1, a_2, b_1, b_2$ with $a_1 \geq b_1 \geq 0$ and $a_2 \geq b_2 \geq 0$.

\begin{definition} \label{graded-like algebra associated FMP}
For non-negative integers $a$ and $b \in \{0, \ldots, a\}$,
we define $\bbQ$-vector spaces
$$
\overline{\cR_{a, b}}:=\cR_{a, b}/\cR_{a, b-1}
$$
and set
$$
\overline{\cR}:=\bigoplus_{a \geq b \geq 0}\overline{\cR_{a,b}}.
$$
Note that $\overline{\cR}$ has a natural $\bbQ$-algebraic structure induced from \eqref{eq:12}.
\end{definition}

\begin{definition} \label{stuffle product}
We define the stuffle product $* : \frH^1 \times \frH^1 \rightarrow \frH^1$ on $\frH^1$ by the following rule and $\bbQ$-bilinearity.
\begin{enumerate}
\renewcommand{\labelenumi}{(\roman{enumi})}

\item $w * 1=1 * w =w$ for all $w \in \frH^1$.

\item $(z_kw_1)*(z_lw_2)=z_k(w_1*z_lw_2)+z_l(z_kw_1* w_2)+z_{k+l}(w_1*w_2)$ for all $w_1, w_2 \in \frH^1$ and $k, l \in \bbZ_{\geq1}$.

\end{enumerate}
For instance, we have
$$
z_2 * z_3=z_2z_3+z_3z_2+z_5.
$$
We define the stuffle product $\bk*\bk'$ of indices $\bk$ and $\bk'$ as the formal sum of indices 
corresponding to $z_{\bk} * z_{\bk'}$. For instance, 
$$
(2)*(3)=(2, 3)+(3, 2)+(5).
$$
\end{definition}

\begin{definition} \label{shuffle-stuffle algebra}
We define a $\bbQ$-vector space $\cS$ by
$$
\cS:=\bigoplus_{\bk, \bk' \in I}\bbQ u_{\bk}v_{\bk'}.
$$
Here $u_{\bk}$ and $v_{\bk'}$ are indeterminates associated to $\bk$ and $\bk'$.
We define the product on $\cS$ by the usual product of $\bbQ$, the stuffle product for $u_\bk$ and the shuffle product for $v_{\bk'}$,
i.e., for $a, b \in \bbQ$ and $\bk_1, \bk_2, \bk'_1, \bk'_2 \in I$, 
$$
(au_{\bk_1}v_{\bk'_1}) \cdot (bu_{\bk_2}v_{\bk'_2}):=ab(u_{\bk_1}*u_{\bk_2})(v_{\bk'_1}\sh v_{\bk'_2}).
$$
Then $(\cS, \cdot)$ is a $\bbQ$-algebra and we call $\cS$ the stuffle-shuffle algebra over $\bbQ$.
\end{definition}

An algebraic interpretation of our main theorem is as follows.

\begin{corollary} \label{algebraic interpretation}
The $\bbQ$-linear homomorphism
$$
\varphi : \cS \rightarrow \overline{\cR} \; ; \;\; u_\bk v_{\bk'} \mapsto \zeta_\cA(\bk)\li_{\bk'}(T) 
$$
is a $\bbQ$-algebra homomorphism.
\end{corollary}

\begin{proof}
Note that FMZVs satisfy the stuffle relation. 
For any $\zeta_\cA(\bk_1)(T^p)^{n_1}\li_{\bk'_1}(T) \in \overline{\cR_{a_1, b_1}}$ and 
$\zeta_\cA(\bk_2)(T^p)^{n_2}\li_{\bk'_2}(T) \in \overline{\cR_{a_2, b_2}}$ 
($n_1, n_2, a_1, a_2, b_1, b_2 \in \bbZ_{\geq0}$ with $a_1\geq b_1 \geq 0$ and $a_2 \geq b_2 \geq 0$), 
we have the following equality by Theorem \ref{MainTheorem} and the stuffle relation of FMZVs:
$$
\zeta_\cA(\bk_1)(T^p)^{n_1}\li_{\bk'_1}(T) \cdot \zeta_\cA(\bk_2)(T^p)^{n_2}\li_{\bk'_2}(T)
=\zeta_\cA(\bk_1 * \bk_2)(T^p)^{n_1+n_2}\li_{\bk'_1 \sh \bk'_2}(T) \in \overline{\cR_{a_1+a_2, b_1+b_2}}.
$$
This completes the proof.
\end{proof}

{\ }\\
\underline{\textbf{Acknowledgement}}\\
The authors would like to thank the members of the KiPAS-AGNT group for giving us a great environment to study and reading the manuscript carefully,
and members of the Department of Mathematics at Keio University for their hospitality.
This research was supported in part by JSPS KAKENHI 21674001, 26247004.

\end{document}